\documentclass[a4paper,11pt,reqno]{amsart}
\usepackage[utf8]{inputenc}
\usepackage[T1]{fontenc}
\usepackage{lmodern}
\usepackage[english]{babel}
\usepackage{amsmath,fullpage}
\usepackage{xfrac}
\usepackage{esint}
\usepackage{stmaryrd,mathrsfs,bm,amsthm,mathtools,yfonts,amssymb,color,braket,booktabs,graphicx,graphics,amsfonts}
\usepackage{latexsym,microtype,indentfirst,hyperref}
\usepackage{xcolor}
\usepackage{courier}
\usepackage[colorinlistoftodos]{todonotes}

\begingroup
\newtheorem{theorem}{Theorem}[section]
\newtheorem{lemma}[theorem]{Lemma}
\newtheorem{proposition}[theorem]{Proposition}
\newtheorem{corollary}[theorem]{Corollary}
\endgroup

\theoremstyle{definition}

\newtheorem{remark}[theorem]{Remark}
\newtheorem{ipotesi}[theorem]{Assumption}

\numberwithin{equation}{section}
\numberwithin{subsection}{section}

\newcommand{\Flat}{\mathbb{F}}
\newcommand{\Mass}{\mathbb{M}}

\newcommand{\MM}{\mathbb{M}_{H}}
\newcommand{\e}{\varepsilon}
\newcommand{\spt}{\mathrm{spt}}
\newcommand{\Lip}{\mathrm{Lip}}

\newcommand{\N}{\mathbb{N}}

\newcommand{\R}{\mathbb{R}}

\newcommand{\Haus}{\mathcal{H}}

\newcommand{\trace}{\mathbin{\vrule height 1.6ex depth 0pt width
0.13ex\vrule height 0.13ex depth 0pt width 1.3ex}}

\newcommand{\D}{\mathscr{D}}
\newcommand{\Rc}{\mathbf{R}}

\newcommand{\F}{\mathbf{F}}

\title[Relaxation of functionals on polyhedral chains]{On the lower semicontinuous envelope of functionals \\ defined on polyhedral chains}

\author{Maria Colombo, Antonio De Rosa, Andrea Marchese, and Salvatore Stuvard}

\newcommand{\Addresses}{{
  \bigskip
  \footnotesize

  M.C., A.D.R., A.M. and S.S., \textsc{Universit\"at Z\"urich, Winterthurerstrasse 190, CH-8057 Z\"urich}
  \par\nopagebreak
  
\bigskip 

\textit{E-mail address}, M.C.: \href{maria.colombo@math.uzh.ch}{maria.colombo@math.uzh.ch}

\medskip

\textit{E-mail address}, A.D.R.: \href{antonio.derosa@math.uzh.ch}{antonio.derosa@math.uzh.ch}

\medskip

\textit{E-mail address}, A.M.: \href{andrea.marchese@math.uzh.ch}{andrea.marchese@math.uzh.ch}

\medskip
 
  \textit{E-mail address}, S.S.: \href{salvatore.stuvard@math.uzh.ch}{salvatore.stuvard@math.uzh.ch}
   
}}

\begin{document}

\begin{abstract}
In this note we prove an explicit formula for the lower semicontinuous envelope of some functionals defined on real polyhedral chains. More precisely, denoting by $H:\R\to[0,\infty)$ an even, subadditive, and lower semicontinuous function with $H(0)=0$, and by $\Phi_H$ the functional induced by $H$ on polyhedral $m$-chains, namely 
$$\Phi_{H}(P) := \sum_{i=1}^{N} H(\theta_{i}) \mathcal{H}^{m}(\sigma_{i}), \quad\mbox{for every }P=\sum_{i=1}^{N} \theta_{i}\llbracket\sigma_{i}\rrbracket\in\mathbf{P}_m(\mathbb{R}^n),$$
we prove that the lower semicontinuous envelope of $\Phi_H$ coincides on rectifiable $m$-currents with the $H$-mass 
$$\mathbb{M}_{H}(R) := \int_E H(\theta(x)) \, d\mathcal{H}^m(x) \quad \mbox{ for every } R=\llbracket E,\tau,\theta\rrbracket \in \mathbf{R}_{m}(\mathbb{R}^{n}).
$$

\vspace{4pt}

\vspace{4pt}

\end{abstract}

\maketitle

\section{Introduction}
Let $H:\R\to[0,\infty)$ be an even, subadditive, and lower semicontinuous function, with $H(0)=0$. The function $H$ naturally induces a functional $\Phi_H$ on the set of polyhedral $m$-chains in $\R^n$, which can be thought as the space  of linear combinations of $m$-simplexes with real coefficients. 
For every polyhedral $m$-chain of the form $P=\sum_{i=1}^{N} \theta_{i}\llbracket\sigma_{i}\rrbracket$ (with non-overlapping $m$-simplexes $\sigma_i$), we set
$$\Phi_{H}(P) := \sum_{i=1}^{N} H(\theta_{i}) \Haus^{m}(\sigma_{i}).$$
It is easy to see that the above assumptions on $H$ are necessary for the functional $\Phi_{H}$ to be (well defined and) lower semicontinuous on polyhedral chains with respect to convergence in flat norm. In this note, we prove that they are also sufficient, and moreover we show that the lower semicontinuous envelope of $\Phi_H$ coincides on rectifiable $m$-currents with the $H$-mass, namely the functional 
$$\MM(R) := \int_E H(\theta(x)) \, d\Haus^m(x), \quad \mbox{ for every rectifiable $m$-current } R=\llbracket E,\tau,\theta\rrbracket.$$

The validity of such a representation has recently attracted some attention. For instance, it is clearly assumed in \cite{Xia} for the choice $H(x)=|x|^\alpha$, with $\alpha\in(0,1)$ , in order to prove some regularity properties of minimizers of problems related to branched transportation (see also \cite{PaoliniStepanov}, \cite{BCM}, \cite{Pegon}) and in \cite{BCF} in order to define suitable approximations of the Steiner problem, with the choice $H(x)=(1+\beta |x|) {\bf 1}_{\R \setminus \{0\}}$, where $\beta>0$ and ${\bf 1}_{A}$ denotes the indicator function of the Borel set $A$. 

We finally remark that the main theorem of the paper (Theorem \ref{cor:F=G} below) is valid in a wider generality. Indeed in \cite[\S 6]{White1999} the author sketches a strategy to prove it in the framework of flat chains with coefficients in a normed abelian group $G$. Motivated by the relevance of such result for real valued flat chains, the ultimate aim of our note is to present a self-contained complete proof of it when $G=\R$.

\subsection*{Acknowledgements}
M. C. acknowledges the support of Dr. Max R\"ossler, of the Walter Haefner Foundation and of the ETH Z\"urich Foundation. A. D.R. is supported by SNF 159403 {\it Regularity questions in geometric measure theory}. A. M. and S. S. are supported by the ERC-grant 306247 {\it Regularity of area minimizing currents}.  

\section{Notation and Main Result}

If $0 \leq m \leq n$, then compactly supported $m$-dimensional currents, rectifiable $m$-currents, polyhedral $m$-chains, and flat $m$-chains in $\R^{n}$ with real coefficients will be denoted $\mathscr{E}_{m}(\R^{n})$, $\Rc_{m}(\R^{n})$, $\mathbf{P}_{m}(\R^{n})$ and $\F_{m}(\R^{n})$, respectively. In what follows, we briefly recall the relevant definitions of the above classes of currents; for the basic definitions about currents, such as the boundary operator $\partial$, the support $\spt$, and the mass norm $\Mass$, we refer the reader to \cite{SimonLN}. Let us denote by $\Lambda^{m}(\R^{n})$ the vector space of $m$-covectors in $\R^{n}$. A current $R$ is in $\Rc_{m}(\R^{n})$ if its action on any differential $m$-form $\omega \in \D^{m}(\R^{n}) := C^{\infty}_{c}(\R^{n}; \Lambda^{m}(\R^{n}))$ can be expressed by
\begin{equation} \label{rect} 
\langle R, \omega \rangle = \int_{E} \langle \omega(x), \tau(x) \rangle \, \theta(x) \, d\Haus^{m}(x),
\end{equation}
where $E \Subset \R^{n}$ is countably $m$-rectifiable, $\tau(x)$ is an $\Haus^{m}$-measurable, unit, simple $m$-vector field orienting the approximate tangent space ${\rm Tan}(E,x)$ at $\Haus^{m}$-a.e. $x \in E$, and $\theta \in L^{1}(\Haus^{m} \trace E ; \left( 0, \infty \right))$ is a positive-valued multiplicity. If $R$ is given by \eqref{rect}, we will write $R = \llbracket E, \tau, \theta \rrbracket$. We remark that the rectifiable currents we are considering all have finite mass and compact support. A polyhedral chain $P \in \mathbf{P}_{m}(\R^{n})$ is a rectifiable current which can be written as a linear combination
\begin{equation} \label{poly}
P = \sum_{i=1}^{N} \theta_{i} \llbracket \sigma_{i} \rrbracket,
\end{equation}
where $\theta_{i} \in \left( 0, \infty \right)$, the $\sigma_{i}$'s are non-overlapping, oriented, $m$-dimensional, convex polytopes (finite unions of $m$-simplexes) in $\R^{n}$ and $\llbracket \sigma_{i} \rrbracket = \llbracket \sigma_{i}, \tau_{i}, 1 \rrbracket$, $\tau_{i}$ being a constant $m$-vector orienting $\sigma_{i}$. If $P \in \mathbf{P}_{m}(\R^{n})$, then its \emph{flat norm} is defined by
\[
\Flat(P) := \inf\lbrace \Mass(S) + \Mass(P - \partial S) \, \colon \, S \in \mathbf{P}_{m+1}(\R^{n}) \rbrace.
\]
Flat $m$-chains can be therefore defined to be the $\Flat$-\emph{completion} of $\mathbf{P}_{m}(\R^{n})$ in $\mathscr{E}_{m}(\R^{n})$. 

We remark that for the spaces of currents considered above the following chain of inclusions holds:
\begin{equation} \label{inclusions}
\mathbf{P}_{m}(\R^{n}) \subset \Rc_{m}(\R^{n}) \subset \F_{m}(\R^{n}) \cap \lbrace T \in \mathscr{E}_{m}(\R^{n}) \, \colon \, \Mass(T) < \infty \rbrace.
\end{equation}

The flat norm $\Flat$ extends to a functional (still denoted $\Flat$) on $\mathscr{E}_{m}(\R^{n})$, which coincides on $\F_{m}(\R^{n})$ with the completion of the flat norm on $\mathbf{P}_{m}(\R^{n})$, by setting:
\begin{equation} \label{flat_corr}
\Flat(T) := \inf\lbrace \Mass(S) + \Mass(T - \partial S) \, \colon \, S \in \mathscr{E}_{m+1}(\R^{n}) \rbrace.
\end{equation}

In the sequel, we will also use the following equivalent characterization of the flat norm of a flat chain (cf. \cite[4.1.12]{FedererBOOK} and \cite[4.5]{Mor}). If $T \in \F_{m}(\R^{n})$ and $K \subset \R^{n}$ is a ball such that $\spt(T) \subset K$, then 
\begin{equation} \label{altra-flat}
\Flat(T) = \sup\lbrace \langle T, \omega \rangle \, \colon \, \omega \in \D^{m}(\R^{n}) \mbox{ with } \| \omega \|_{C^{0}(K; \Lambda^{m}(\R^{n}))} \leq 1, \, \|d \omega \|_{C^{0}(K; \Lambda^{m+1}(\R^{n}))} \leq 1 \rbrace.
\end{equation}

\begin{ipotesi} \label{hp:funzione_H}
In what follows, we will consider a Borel function $H \colon \R \to \left[ 0, \infty \right)$ satisfying the following hypotheses:
\begin{itemize}
\item[$(H1)$] $H(0) = 0$ and $H$ is even, namely $H(-\theta) = H(\theta)$ for every $\theta \in \R$;

\item[$(H2)$] $H$ is subadditive, namely $H(\theta_{1} + \theta_{2}) \leq H(\theta_{1}) + H(\theta_{2})$ for every $\theta_{1}, \theta_{2} \in \R$;

\item[$(H3)$] $H$ is lower semicontinuous, namely $H(\theta) \leq \liminf_{j \to \infty} H(\theta_{j})$ whenever $\theta_{j}$ is a sequence of real numbers such that $| \theta - \theta_{j} | \searrow 0$ when $j \uparrow \infty$.
\end{itemize}
\end{ipotesi}

\begin{remark} \label{numerabile_subadd}
Observe that the hypotheses $(H2)$ and $(H3)$ imply that $H$ is in fact countably subadditive, namely
\[
H\left( \sum_{j=1}^{\infty} \theta_{j}\right) \leq \sum_{j=1}^{\infty} H(\theta_{j}),
\]
for any sequence $\{ \theta_{j} \}_{j=1}^{\infty} \subset \R$ such that $\sum_{j=1}^{\infty} \theta_{j}$ converges.
\end{remark}

\begin{remark} \label{monotonia}
Let $\tilde{H} \colon \left[ 0, \infty \right) \to \left[ 0, \infty \right)$ be any Borel function satisfying:
\begin{itemize}
\item[$(\tilde{H}1)$] $\tilde{H}(0) = 0$;

\item[$(\tilde{H}2)$] $\tilde{H}$ is subadditive and monotone non-decreasing, i.e. $\tilde{H}(\theta_{1}) \leq \tilde{H}(\theta_{2})$ for any $0 \leq \theta_{1} \leq \theta_{2}$;

\item[$(\tilde{H}3)$] $\tilde{H}$ is lower semicontinuous,
\end{itemize}
and let $H \colon \R \to \left[ 0, \infty \right)$ be the even extension of $\tilde{H}$, that is set $H(\theta) := \tilde{H}(|\theta|)$ for every $\theta \in \R$. Then, the function $H$ satisfies Assumption \ref{hp:funzione_H}. 
\end{remark}

Let $H$ be as in Assumptions \ref{hp:funzione_H}. We define a functional $\Phi_{H} \colon \mathbf{P}_{m}(\R^{n}) \to \left[ 0, \infty \right)$ as follows. Assume $P \in \mathbf{P}_{m}(\R^{n})$ is as in \eqref{poly}. Then, we set
\begin{equation} \label{eq:funzionale_Phi}
\Phi_{H}(P) := \sum_{i=1}^{N} H(\theta_{i}) \Haus^{m}(\sigma_{i}). 
\end{equation} 

The functional $\Phi_{H}$ naturally extends to a functional $\MM$, called the $H$-mass, defined on $\Rc_{m}(\R^{n})$ by
\begin{equation}\label{eq:funzionale_G}
\MM(R) := \int_E H(\theta(x)) \, d\Haus^m(x), \quad \mbox{ for every } R=\llbracket E,\tau,\theta\rrbracket \in \Rc_{m}(\R^{n}).
\end{equation}

We also define the functional $F_{H} \colon \F_{m}(\R^{n}) \to \left[ 0, \infty \right]$ to be the lower semicontinuous envelope of $\Phi_{H}$. More precisely, for every $T \in \F_{m}(\R^{n})$ we set
\begin{equation} \label{eq:funzionale_F}
F_{H}(T) := \inf\bigg\lbrace \liminf_{j \to \infty} \Phi_{H}(P_{j}) \, \colon \, P_{j} \in \mathbf{P}_{m}(\R^{n}) \mbox{ with } \Flat(T - P_{j}) \searrow 0\bigg\rbrace.
\end{equation}

The main result of the paper is the following theorem.

\begin{theorem} \label{cor:F=G}
Let $H$ satisfy Assumption \ref{hp:funzione_H}. Then, $F_{H} \equiv \MM$ on $\Rc_{m}(\R^{n})$.
\end{theorem}

In order to prove Theorem \ref{cor:F=G}, we adopt the following strategy. First, we show that the functional $\MM$ is lower semicontinuous on rectifiable currents, with respect to the flat convergence, as in the following proposition, with $A = \R^{n}$. If $T = \llbracket E, \tau, \theta \rrbracket$ and $B \subset \R^{n}$ is a Borel set, we denote the \emph{restriction} of $T$ to $B$ by setting $T \trace B := \llbracket E \cap B, \tau, \theta \rrbracket \in \Rc_{m}(\R^{n})$. The restriction operator analogously extends to all currents which can be represented by integration.

\begin{proposition}\label{prop:lsc}
Let $H$ satisfy Assumption \ref{hp:funzione_H}, and let $A \subset \R^{n}$ be open. Let $T_{j}, T \in \Rc_{m}(\R^{n})$ be rectifiable $m$-currents such that $\Flat(T - T_{j}) \searrow 0$ as $j \to \infty$. Then
\begin{equation} \label{eq:lsc}
\MM(T \trace A) \leq \liminf_{j \to \infty} \MM(T_{j} \trace A).
\end{equation}  
\end{proposition}


Next, we observe that, as an immediate consequence of Proposition \ref{prop:lsc} and of the properties of the lower semicontinuous envelope, it holds
\begin{equation} \label{eq:disug_scema}
\MM(R) \leq F_{H}(R) \quad \mbox{ for every } R \in \Rc_{m}(\R^{n}).
\end{equation}

The opposite inequality, which completes the proof of Theorem \ref{cor:F=G}, is obtained as a consequence of the following proposition, which provides a polyhedral approximation in flat norm of any rectifiable $m$-current $R$ with $H$-mass and mass close to those of the given $R$.

\begin{proposition} \label{thm:poly_approx}
Let $H$ be any Borel function satisfying $(H1)$ in Assumption \ref{hp:funzione_H}, and let $R \in \Rc_{m}(\R^{n})$ be rectifiable. For every $\varepsilon > 0$ there exists a polyhedral $m$-chain $P \in \mathbf{P}_{m}(\R^{n})$ such that
\begin{equation} \label{eq:poly_approx}
\Flat(R - P)\leq\e, \qquad \Phi_{H}(P) \leq \MM(R) + \varepsilon \qquad \mbox{and} \qquad \Mass(P) \leq \Mass(R)+ \varepsilon.
\end{equation} 
\end{proposition}

Theorem \ref{cor:F=G} characterizes the lower semicontinuous envelope $F_{H}$ on rectifiable currents to be the (possibly infinite) $H$-mass $\MM$. Without further assumptions on $H$, the lower semicontinuous envelope $F_{H}$ can have finite values on flat chains which are non-rectifiable (for instance, the choice $H(\theta):=|\theta|$ induces the mass functional $F_H= \Mass$). If instead we add the natural hypothesis that $H$ is monotone non-decreasing on $\left[ 0, \infty \right)$, then there is a simple necessary and sufficient condition which prevents this to happen in the case of flat chains with finite mass, thus allowing us to obtain an explicit representation for $F_H$ on all flat chains with finite mass.

\begin{proposition}\label{thm:normal}
Let $H$ be as in Assumption \ref{hp:funzione_H} and monotone non-decreasing on $\left[ 0, \infty \right)$. The condition
\begin{equation}\label{ratio}
\lim_{\theta \searrow 0^{+}} \frac{H(\theta)}{\theta}=+\infty.
\end{equation}
holds if and only if 
\begin{equation}\label{non-rect}
F_H(T)=\begin{cases}
\MM(T) &\mbox{for }T \in \Rc_m(\R^n),
\\
+\infty  & \mbox{for } T \in (\F_m(\R^n)\cap  \lbrace T \in \mathscr{E}_{m}(\R^{n}) \, \colon \, \Mass(T) < \infty \rbrace) \setminus \Rc_m(\R^n).
\end{cases}
\end{equation}
\end{proposition}

\section{Proof of Proposition \ref{prop:lsc}}\label{s:sci}
This section is devoted to the proof of Proposition \ref{prop:lsc}. It is carried out by slicing the rectifiable currents $T_j$ and $T$ and reducing the proposition to the lower semicontinuity of $0$-dimensional currents. Some of the techniques here adopted are borrowed from \cite[Lemma 3.2.14]{depauwhardt}.

We recall some preliminaries on the slicing of currents. 
Given $m \leq n$, let $I(n,m)$ be the set of $m$-tuples $(i_1,\ldots,i_m)$ with
$$1\leq i_1<\ldots<i_m\leq n.$$

Let $\{e_1,\ldots,e_n\}$ be an orthonormal basis of $\R^n$. For any $I=(i_1,\ldots,i_m)\in I(n,m)$, let $V_{I}$ be the $m$-plane spanned by $\{e_{i_1},\ldots,e_{i_m}\}$. Given an $m$-plane $V$, we will denote $p_V$ the orthogonal projection onto $V$. If $V=V_I$ for some $I$, we write $p_I$ in place of $p_{V_I}$. Given a current $T\in \F_{m}(\R^{n})$, a Lipschitz function $f:\R^n\to\R^k$ for some $k\leq m$ and $y\in\R^k$, we denote by $\langle T,f,y\rangle$ the $(m-k)$-dimensional \emph{slice} of $T$ in $f^{-1}(y)$ (see \cite[Section 4.3]{FedererBOOK}). Intuitively, this can be thought as the ``intersection'' of the current $T$ with the level set $f^{-1}(y)$.

Let us denote by $Gr(n,m)$ the Grassmannian of $m$-dimensional planes in $\R^n$, and by $\gamma_{n,m}$ the Haar measure on $Gr(n,m)$ (see \cite[Section 2.1.4]{KrantzParks}).

In the following lemma, we prove a version of the integral-geometric equality for the $H$-mass, which is a consequence of \cite[3.2.26;
2.10.15]{FedererBOOK} (see also \cite[(21)]{depauwhardt}). We observe that the hypotheses $(H2)$ and $(H3)$ on the function $H$ are not needed here, and indeed Lemma \ref{maestosa-integral-geo} below is valid for any Borel function $H$ for which the $H$-mass $\MM$ is well defined.

\begin{lemma}\label{maestosa-integral-geo}
Let $E \subseteq \R^n$ be $m$-rectifiable. Then there exists $c=c(n,m)$ such that the following integral-geometric equality holds:
\begin{equation}
\label{eqn:quel-che-dice-Fed}
\Haus^m(E) = c\int_{Gr(n,m)}\int_{\R^m} \Haus^{0}( p_{V}^{-1}(\{y\}) \cap E)   \, d\Haus^m(y)\, d\gamma_{n,m}(V).
\end{equation} 
In particular, if $R \in \Rc_{m}(\R^{n})$, 
\begin{equation}\label{e:int_geom}
\MM(R)=c\int_{Gr(n,m)\times\R^m}\MM \big(\langle R,p_V,y\rangle \big) d(\gamma_{n,m}\otimes\Haus^m)(V,y).
\end{equation}
\end{lemma}
\begin{proof}
The equality \eqref{eqn:quel-che-dice-Fed} is proved in \cite[3.2.26; 2.10.15]{FedererBOOK}.
For any Borel set $A \subset \R^{n}$, denoting $f = \mathbf{1}_{A}$, \eqref{eqn:quel-che-dice-Fed} implies that
\begin{equation}\label{e:int_geom-f}
\int_{E} f(x)\,d\Haus^m(x) = c\int_{Gr(n,m)}\int_{\R^m} \int_{E} f(x) \, \mathbf{1}_{  p_{V}^{-1}(\{y\})}(x) \, d\Haus^0(x)  \, d\Haus^m(y)\, d\gamma_{n,m}(V).
\end{equation}
Since the previous equality is linear in $f$, it holds also when $f$ is piecewise constant. Since the measure $\Haus^m \trace E$ is $\sigma$-finite, the equality can be extended to any measurable function $f\in L^1(\Haus^{m} \trace E)$. The case $f \notin L^{1}(\Haus^{m} \trace E)$ follows from the Monotone Convergence Theorem via a simple truncation argument.

Taking $R= \llbracket E, \theta, \tau \rrbracket$, and applying \eqref{e:int_geom-f} with $f(x)= H(\theta(x))$, we deduce that 
$$\MM(R) =
 c\int_{Gr(n,m)}\int_{\R^m} \int_{E\cap  p_{V}^{-1}(\{y\})} H(\theta(x)) \, d\Haus^0(x)  \, d\Haus^m(y)\, d\gamma_{n,m}(V).
$$
We observe that the right-hand side coincides with the right-hand side in \eqref{e:int_geom}
since for $\Haus^{m}$-a.e. $y \in \R^{m}$ the $0$-dimensional current $\langle R,p_V,y\rangle $ is concentrated on the set $E \cap p_{V}^{-1}(y)$ and  
 its density at any $x\in E \cap p_{V}^{-1}(y)$ is $\theta(x)$.
\end{proof}

We prove the lower semicontinuity in \eqref{eq:lsc} by an explicit computation in the case $m=0$. Then, by slicing, we get the proof for $m > 0$, too.
\begin{proof}[Proof of Proposition \ref{prop:lsc}]
	{\it Step 1: the case $m=0$.} Let $T_{j} := \llbracket E_{j}, \tau_{j}, \theta_{j}\rrbracket, T := \llbracket E, \tau, \theta \rrbracket \in \Rc_{0}(\R^{n})$ be such that $\Flat(T - T_{j}) \searrow 0$ as $j \to \infty$. Since $T \trace A$ is a signed, atomic measure, we write
	$$T \trace A = \sum_{i\in \N} \tau(x_{i}) \theta(x_{i}) \delta_{x_i}$$
	for distinct points $\{x_i\}_{i\in \N}  \subseteq E \cap A$, orientations $\tau(x_{i}) \in \{-1, 1 \}$, and for $\theta(x_{i}) > 0$.
	Fix $\e>0$ and let $N = N(\e) \in \mathbb{N}$ be such that
\begin{equation}
	\label{eqn:finite}
		\MM(T \trace A)-  \sum_{i = 1}^{N} H(\theta(x_{i})) \leq \e \qquad \mbox{if }\MM(T \trace A)<\infty
\end{equation}
and 
\begin{equation}
\label{eqn:finite-infin}
 \sum_{i = 1}^{N} H(\theta(x_{i})) \geq \frac 1\e \qquad \mbox{otherwise}.
\end{equation}
Since $H$ is positive, even, and lower semicontinuous, for every $i \in \{1,\dots,N\}$ it is possible to determine $\eta_{i} = \eta_{i}(\e, \theta(x_{i})) > 0$ such that
	\begin{equation} \label{semi-cont}
	H(\theta) \geq (1 - \e) H(\theta(x_{i})) \quad \mbox{ for every } |\theta - \tau(x_{i}) \theta(x_{i})| < \eta_{i}.
	\end{equation}
Moreover, for every $i \in \{1,\dots,N\}$ there exists $0 < r_{i} < \min\{{\rm dist}(x_{i}, \partial A), 1\}$ such that the balls $B(x_{i}, r_{i})$ are pairwise disjoint, and moreover such that for every $\rho \leq r_{i}$ it holds
\begin{equation} \label{conv_molteplic1}
\left|\tau(x_{i})\theta(x_{i}) - \sum_{x \in E \cap B(x_{i}, \rho)} \tau(x) \theta(x) \right| \leq \frac{\eta_{i}}{2}.
\end{equation}
Our next aim is to prove that in sufficiently small balls and for $j$ large enough, the sum of the multiplicities of $T_j$ (with sign) is close to the sum of the multiplicities of $T$. In order to do this, we would like to test the current $T-T_j$ with the indicator function of each ball. Since this test is not admissible, we have to consider a smooth and compactly supported extension of it outside the ball, provided we can prove that the flat convergence of $T_j$ to $T$ localizes to the ball. From this, our claimed convergence of the signed multiplicities follows by the characterization of the flat norm in \eqref{altra-flat}.

 To make this formal, we define $\eta_{0} := \min_{1 \leq i \leq N} \eta_{i}$ and $r_{0} := \min_{1 \leq i \leq N} r_{i}$. Let $j_{0}$ be such that 
\[
\Flat(T - T_{j}) \leq \frac{\eta_{0} r_{0}}{16} \quad \mbox{ for every } j \geq j_{0}.
\]	
By the definition \eqref{flat_corr} of flat norm, there exist $R_{j} \in \mathscr{E}_{0}(\R^{n})$, $S_{j} \in \mathscr{E}_{1}(\R^{n})$ such that $T - T_{j} = R_{j} + \partial S_{j}$ with $\Mass(R_{j}) + \Mass(S_{j}) \leq \frac{\eta_{0} r_{0}}{8}$ for every $j \geq j_{0}$.  Observe that the mass and the mass of the boundary of both $R_{j}$ and $S_{j}$ are finite, and thus by \cite[4.1.12]{FedererBOOK} it holds $R_{j} \in \F_{0}(\R^{n})$ and $S_{j} \in \F_{1}(\R^{n})$. We want to deduce that for every $i \in \{ 1,\dots,N \}$ there exists $\rho_{i} \in \left( \frac{r_{0}}{2}, r_{0} \right)$ such that 
\[
\Flat((T - T_{j}) \trace B(x_{i}, \rho_{i})) \leq \frac{\eta_{0}}{2}.
\] 
Indeed, for any fixed $i \in \{ 1,\dots,N\}$ one has that for a.e. $\rho \in \left( \frac{r_{0}}{2}, r_{0} \right)$
\begin{equation} \label{localized_flat}
\begin{split}
(T - T_{j}) \trace B(x_{i}, \rho) &= R_{j} \trace B(x_{i}, \rho) + (\partial S_{j}) \trace B(x_{i}, \rho) \\
&= R_{j} \trace B(x_{i}, \rho) - \langle S_{j}, {\rm d}(x_{i}, \cdot), \rho \rangle + \partial\left( S_{j} \trace B(x_{i},\rho) \right),
\end{split}
\end{equation}
where ${\rm d}(x_{i},z) := |x_{i} - z|$ and where the last identity holds by the definition of slicing for currents with finite mass and finite mass of the boundary, the so called \emph{normal currents} (cf. \cite[4.2.1]{FedererBOOK}). On the other hand, by \cite[4.2.1]{FedererBOOK} we have
\[
\int_{\frac{r_{0}}{2}}^{r_{0}} \Mass(\langle S_{j}, {\rm d}(x_{i}, \cdot), \rho \rangle) \, d\rho \leq \Mass(S_{j} \trace (B(x_{i}, r_0) \setminus B(x_{i}, \frac{r_{0}}{2}))) \leq \frac{\eta_{0}r_{0}}{8}.
\]	
Hence, there exists $\rho_{i} \in \left( \frac{r_{0}}{2}, r_{0} \right)$ such that 
\begin{equation} \label{massa_slice}
\Mass(\langle S_{j}, {\rm d}(x_{i}, \cdot), \rho_{i} \rangle) \leq \frac{\eta_{0}}{4}.
\end{equation}
We conclude from \eqref{localized_flat} that 
\begin{equation} \label{localized_flat2}
\begin{split}
\Flat((T - T_{j}) \trace B(x_{i}, \rho_{i})) &\leq \Mass(R_{j} \trace B(x_{i}, \rho_{i})) + \Mass(\langle S_{j}, {\rm d}(x_{i}, \cdot), \rho_{i} \rangle) + \Mass( S_{j} \trace B(x_{i}, \rho_{i})) \\
&\overset{\eqref{massa_slice}}{\leq} \frac{\eta_{0} r_{0}}{4} + \frac{\eta_{0}}{4} \leq \frac{\eta_{0}}{2}. 
\end{split}
\end{equation} 	
	
Using the characterization of the flat norm in \eqref{altra-flat}, and testing the currents $(T - T_{j}) \trace B(x_{i}, \rho_{i})$ with any smooth and compactly supported function $\phi_{i} \colon \R^{n} \to \R$ which is identically $1$ on $B(x_{i}, \rho_{i})$, we obtain
\begin{equation} \label{conv_molteplic2}
\left| \sum_{x \in E \cap B(x_{i}, \rho_{i})} \tau(x) \theta(x) - \sum_{y \in E_{j} \cap B(x_{i}, \rho_{i})} \tau_{j}(y) \theta_{j}(y) \right| \leq \frac{\eta_{0}}{2}.
\end{equation}
	
Combining \eqref{conv_molteplic2} with \eqref{conv_molteplic1}, we deduce by triangle inequality that
\begin{equation} \label{conv_molt_fin}
\left| \tau(x_{i}) \theta(x_{i}) - \sum_{y \in E_{j} \cap B(x_{i}, \rho_{i})} \tau_{j}(y) \theta_{j}(y) \right| \leq \eta_{i}.
\end{equation}	

Finally, using \eqref{semi-cont} and the fact that $H$ is countably subadditive (cf. Remark \ref{numerabile_subadd}), we conclude that for every $j \geq j_{0}$
\[
\begin{split}
H(\theta(x_{i})) &\leq \frac{1}{1 - \e} H\left( \sum_{y \in E_{j} \cap B(x_{i}, \rho_{i})} \tau_{j}(y) \theta_{j}(y) \right) \\
&\leq \frac{1}{1 - \e} \sum_{y \in E_{j} \cap B(x_{i}, \rho_{i})} H(\theta_{j}(y)) \\
&= \frac{1}{1 - \e} \MM(T_{j} \trace B(x_{i}, \rho_{i})).
\end{split}
\]

	Summing over $i$, since the balls $B(x_i,\rho_{i})$ are pairwise disjoint, we get that 
	$$ \sum_{i = 1}^{N} H(\theta_i) \leq \frac{1}{1-\e} \liminf_{j \to \infty} \sum_{i=1}^N\MM(T_j\trace  B(x_i,\rho_{i})) \leq \frac{1}{1-\e} \liminf_{j \to \infty} \MM(T_j \trace A).
	$$
	By \eqref{eqn:finite} (or \eqref{eqn:finite-infin} in the case that $\MM(T \trace A)=\infty$) and since $\e$ is arbitrary, we find \eqref{eq:lsc}.\\

	{\it Step 2 (Reduction to $m=0$ through integral-geometric equality).} We prove now  Proposition~\ref{prop:lsc} for $m > 0$.
	Up to subsequences, we can assume
	$$\lim_{j \to \infty} \MM (T_j \trace A)= \liminf_{j \to \infty} \MM (T_j \trace A) .$$
	
By \cite[4.3.1]{FedererBOOK}, for every $ V \in Gr(n,m)$ it holds
\begin{equation}
\label{eqn:flat-slices}
\int_{\R^m} \Flat(\langle T_{j} - T, p_V, y\rangle) \, dy \leq \Flat (T_{j} - T),
\end{equation}
	 
	Integrating the inequality \eqref{eqn:flat-slices} in $V\in Gr(n,m)$ and using that $\gamma_{n,m}$ is a probability measure on $Gr(n,m)$  we get
	$$\lim_{j \to \infty} \int_{Gr(n,m)\times \R^m}\Flat(\langle T_j-T,p_V,y\rangle) d(\gamma_{n,m} \otimes \Haus^m)(V,y) \leq \lim_{j \to \infty} \Flat(T_j-T) = 0.$$
	Since the integrand $\Flat(\langle T_j-T,p_V,y\rangle)$ is converging to $0$ in $L^1$, up to subsequences, we get
	$$\lim_{j \to \infty}\Flat(\langle T_j-T,p_V,y\rangle)  = 0 \qquad 
	\mbox{for $\gamma_{n,m} \otimes \Haus ^m$-a.e. $(V, y) \in Gr(n,m)\times\R^m$}.$$
	We conclude from Step 1 that
\begin{equation} \label{step1}	
	\MM(\langle T,p_{V},y\rangle \trace A) \leq  \liminf_{j \to \infty} \MM (\langle T_j,p_{V},y\rangle \trace A) \qquad 
	\mbox{for $\gamma_{n,m} \otimes \Haus ^m$-a.e. $(V, y) \in Gr(n,m)\times\R^m$}.
\end{equation}

By \cite[(5.15)]{Ambrosio2000}, for every $V \in Gr(n,m)$ one has $\langle T, p_{V}, y \rangle \trace A = \langle T \trace A, p_{V}, y \rangle$ for $\Haus^{m}$-a.e. $y \in \R^{m}$.
	
	In order to conclude, we apply twice the integral-geometric equality \eqref{e:int_geom}. Indeed, using \eqref{step1} and Fatou's lemma, we get
	\begin{equation}
	\begin{split}
	\MM(T \trace A)&=c\int_{Gr(n,m)\times\R^m}\MM\big(\langle T \trace A,p_V,y\rangle\big) d(\gamma_{n,m}\otimes\Haus^m)(V,y)\\
	&\leq c\int_{Gr(n,m)\times\R^m}\liminf_{j\to \infty}\MM \big(\langle T_j \trace A ,p_V,y\rangle \big) d(\gamma_{n,m}\otimes\Haus^m)(V,y)\\
	&\leq c\liminf_{n\to \infty} \int_{Gr(n,m)\times\R^m}\MM \big(\langle T_j \trace A ,p_V,y\rangle \big) d(\gamma_{n,m}\otimes\Haus^m)(V,y)\\
	&=\liminf_{j\to \infty}\MM(T_j \trace A).
	\end{split}
	\end{equation}
This concludes the proof of Step 2, so the proof of Proposition~\ref{prop:lsc} is complete.
\end{proof}

\section{Proof of Proposition \ref{thm:poly_approx}}
In order to prove the proposition, we will consider a family of pairwise disjoint balls which contain the entire mass of the current $R$, up to a small error. Then, we replace in any of these balls the current $R$ with an $m$-dimensional disc with constant multiplicity. Afterwards, we further approximate each disc with polyhedral chains.


We begin with the following lemma, where we prove that, at many points $x$ in the $m$-rectifiable set supporting the current $R$ and at sufficiently small scales (depending on the point), $R$ is close in the flat norm to the tangent $m$-plane at $x$ weighted with the multiplicity of $R$ at $x$. 

In this section, given the $m$-current $R = \llbracket E, \tau, \theta \rrbracket$, for a.e. $x\in E$ we denote with $\pi_{x}$ the affine $m$-plane through $x$ spanned by the (simple) $m$-vector $\tau(x)$ and with $S_{x,\rho}$ the $m$-current
$$S_{x,\rho} := \llbracket B(x,\rho) \cap \pi_{x}, \tau(x), \theta(x) \rrbracket.$$

\begin{lemma} \label{lemma_tecnico}
Let $\e > 0$, and let $R = \llbracket E, \tau, \theta \rrbracket$ be a rectifiable $m$-current in $\R^{n}$. There exists a subset $E' \subset E$ such that the following holds:
\begin{itemize}

\item[$(i)$] $\Mass(R \trace (E \setminus E')) \leq \e$;

\item[$(ii)$] for every $x \in E'$ there exists $r = r(x) > 0$ such that for any $0 < \rho \leq r$
\begin{equation}
\label{eqn:flat-blow-up}
\Flat(R \trace (E' \cap B(x,\rho)) - S_{x,\rho}) \leq \e \Mass(R \trace B(x,\rho)).
\end{equation}
\end{itemize}
\end{lemma}

\begin{proof}
Since $E$ is countably $m$-rectifiable, there exist countably many linear $m$-dimensional planes $\Pi_{i}$ and $C^{1}$ and globally Lipschitz maps $f_{i} \colon \Pi_{i} \to \Pi_{i}^{\perp}$ such that
\[
E \subset E_{0} \cup \bigcup_{i=1}^{\infty} {\rm Graph}(f_{i}),
\]
with $\Haus^{m}(E_{0}) = 0$. We will denote $\Sigma_{i} := {\rm Graph}(f_{i}) \subset \R^{n}$. For every $x \in \bigcup_{i=1}^{\infty} \Sigma_{i}$, we let $i(x)$ be the first index such that $x \in \Sigma_{i}$. Then, for every $i \geq 1$, we define $R_{i} := \llbracket E \cap \Sigma_{i}, \tau, \theta_{i} \rrbracket$, where 
\begin{equation} \label{multi}
\theta_{i}(x) :=
\begin{cases}
\theta(x) &\mbox{ if $i = i(x) $ } \\
0 &\mbox{ otherwise}.
\end{cases} 
\end{equation}
Clearly, $R = \sum_{i=1}^{\infty} R_{i}$ and $\Mass(R) = \sum_{i=1}^{\infty} \Mass(R_{i})$. Hence, there exists $N = N(\e)$ such that 
\begin{equation} \label{coda_piccola}
\sum_{i\geq N+1} \Mass(R_{i}) \leq \e.
\end{equation}

Now, recall that $x$ is a Lebesgue point of the function $\theta_i$ with respect to the Radon measure $\Haus^m \trace \Sigma_i$ if
\[
\lim_{r \to 0} \frac{1}{\Haus^m(\Sigma_i \cap B(x,r)) }\int_{\Sigma_i \cap B(x,r)} |\theta_i(y)-\theta_i(x)| \, d \Haus^m(y)  = 0
.
\]

We define the set $E' \subset E$ by
\begin{equation} \label{insieme_buono}
\begin{split}
E' := \bigg\lbrace x \in E \cap \bigcup_{i=1}^{N} \Sigma_{i} &\quad \mbox{such that $x$ is a Lebesgue point of $\theta_{i}$} \\
&\mbox{with respect to $\Haus^{m} \trace \Sigma_{i}$ for every $i \in \{1,\dots, N\}$} \bigg\rbrace,
\end{split}
\end{equation}
and we observe that $(i)$ follows from \eqref{coda_piccola} and \cite[Corollary 2.23]{AFP}.

Let us set 
\begin{equation} \label{Lip}
L := \max\lbrace \Lip(f_{i}) \, \colon \, i=1,\dots,N\rbrace.
\end{equation}

Fix $i \in \{1, \dots, N \}$. For every $x \in \Sigma_{i}$ there exists $r > 0$ such that whenever $j \in \{1, \dots, N \}$ is such that $\Sigma_{j} \cap B(x,\sqrt{n} r) \neq \emptyset$, then $x \in \Sigma_{j}$.

Now, fix any point $x \in E'$, and fix an index $j \in \{ 1, \dots, N\}$ such that $x \in \Sigma_{j}$. If $j = i(x)$, then $\theta_{j}(x) = \theta(x) > 0$. Since by the definition of $E'$
\begin{equation} \label{density}
\lim_{r \to 0} \frac{\Mass(R_{j} \trace (\Sigma_{j} \cap B(x,r)))}{\Haus^{m}(\Sigma_{j} \cap B(x,r))} = \theta_{j}(x),
\end{equation}
then there exists $r > 0$ such that for any $0 < \rho \leq \sqrt{n} r$ 
\begin{equation} \label{density2}
\frac{\Mass(R_{j} \trace (\Sigma_{j} \cap B(x,\rho)))}{\Haus^{m}(\Sigma_{j} \cap B(x,\rho))} \geq \frac{\theta_{j}(x)}{2}.
\end{equation}
Again by \cite[Corollary 2.23]{AFP} applied with $\mu = \Haus^{m} \trace \Sigma_{j}$ and $f = \theta_{j}$, there exists a radius $r > 0$ (depending on $x$) such that 
\begin{equation} \label{Lebesgue_pt}
\begin{split}
\int_{\Sigma_{j} \cap B(x,\rho)} |\theta_{j}(y) - \theta_{j}(x)| \, d\Haus^{m}(y) &\leq \e \frac{\theta_{j}(x)}{2} \Haus^{m}(\Sigma_{j} \cap B(x,\rho)) \\
&\leq \e \frac{\Mass(R_{j} \trace (\Sigma_{j} \cap B(x,\rho)))}{\Haus^{m}(\Sigma_{j} \cap B(x,\rho))} \Haus^{m}(\Sigma_{j} \cap B(x,\rho)) \\
&\leq \e \Mass(R_{j} \trace B(x,\rho)),
\end{split}
\end{equation} 
for every $0 < \rho \leq \sqrt{n} r$.

If, instead, $j \neq i(x)$, then $\theta_{j}(x) = 0$ and therefore there exists a radius $r > 0$ (depending on $x$) such that for every $0 < \rho \leq \sqrt{n} r$
\begin{equation} \label{altro_caso}
\begin{split}
\int_{\Sigma_{j} \cap B(x,\rho)} \theta_{j}(y) \, d\Haus^{m}(y) &\leq \frac{\e \theta_{i(x)}(x)}{N(1+L)^{m}} \Haus^{m}(\Sigma_{j} \cap B(x, \rho)) \\
&\leq \frac{\e}{N} \theta_{i(x)}(x) \omega_{m} \rho^{m} \\
&\overset{\eqref{density2}}{\leq} 2 \frac{\e}{N} \Mass(R_{i(x)} \trace B(x, \rho)),
\end{split}
\end{equation}
where $\omega_{m}$ denotes the volume of the unit ball in $\R^{m}$.

Fix any point $x \in E'$ and let $i = i(x)$. By possibly reparametrizing $f_i|_{\Pi_{i} \cap B(x, r)}$ from the $m$-plane tangent to $\Sigma_{i}$ at $x$, translating and tilting such a plane, we can assume that $x = 0$,
 $\Pi_{i} = \lbrace x_{m+1} = \dots = x_{n} = 0 \rbrace$ and $\nabla f_{i}(x) = 0$. By possibly choosing a smaller radius $r = r(x) > 0$, we may also assume that 
\begin{equation} \label{poco_tilt}
|\nabla f_{i}| \leq \e \quad \mbox{in } \Pi_{i} \cap B(x, r).
\end{equation}
With these conventions, the current $S_{x,\rho}$ in the statement reads $S_{x,\rho} = \llbracket B(0,\rho) \cap \Pi_{i}, \tau(0), \theta_{i}(0) \rrbracket$. We let $F_{i} \colon \Pi_{i} \times \Pi_{i}^{\perp} \to \R^{n}$ be given by $F_{i}(z,w) := \left( z, f_{i}(z) \right)$, and we set $\tilde{R}_{i} := (F_{i})_{\sharp} S_{x,\rho} \in \Rc_{m}(\R^{n})$.  

By \eqref{poco_tilt} and the homotopy formula (cf. \cite[26.23]{SimonLN}) applied with $g = F_{i}$ and $f(z,w) := (z,0)$, we have, denoting $C(x,\rho) := (B(x,\rho) \cap \Pi_{i}) \times \Pi_{i}^{\perp}$,
\begin{equation} \label{stima1}
\begin{split}
\Flat(\tilde{R}_{i} - S_{x,\rho}) &\leq C \| g - f \|_{L^{\infty}(C(x, \rho))} \left( \Mass(S_{x, \rho}) + \Mass(\partial S_{x, \rho}) \right) \\
&\leq C \e \rho \left( \Mass(S_{x,\rho}) + \Mass(\partial S_{x,\rho}) \right) \\
&\leq C \e \theta(x) \omega_{m} \rho^{m} \\
&\leq C \e \theta(x) \Haus^{m}(\Sigma_{j} \cap B(x, \rho)) \\
&\overset{\eqref{density2}}{\leq} C \e \Mass(R_{i} \trace B(x,\rho)).
\end{split}
\end{equation}

Now, observe that, if we denote by $\xi_{i}$ the orientation of $\Sigma_{i}$ induced by the orientation of $\Pi_{i} \times \Pi_{i}^{\perp}$ via $F_{i}$, the rectifiable current $\tilde{R}_{i}$ reads $\tilde{R}_{i} = \llbracket \Sigma_{i} \cap C(x, \rho), \xi_{i}, \theta_{i}(x) \rrbracket$ (cf. \cite[27.2]{SimonLN}). Therefore, we can compute
\begin{equation} \label{stima2}
\begin{split}
\Mass(R_{i} \trace B(x,\rho) - \tilde{R}_{i}) &\leq \Mass(R_{i} \trace B(x, \rho) - \tilde{R}_{i} \trace B(x,\rho)) + \Mass(\tilde{R}_{i} \trace (C(x,\rho) \setminus B(x,\rho))) \\
&\overset{\eqref{Lebesgue_pt}}{\leq} \e \Mass(R_{i} \trace B(x,\rho)) + \Mass(\tilde{R}_{i} \trace (C(x,\rho) \setminus B(x,\rho))) \\
&\overset{\eqref{poco_tilt}}{\leq} \e \Mass(R_{i} \trace B(x,\rho)) + C \e \theta_{i}(x) \Haus^{m}(\Sigma_{i} \cap B(x,\rho)) \\
&\overset{\eqref{density2}}{\leq} C \e \Mass(R_{i} \trace B(x,\rho)).
\end{split}
\end{equation}

Hence, we conclude:
\begin{equation} \label{stima_fin}
\begin{split}
\Flat(R \trace E' \cap &B(x,\rho) - S_{x,\rho}) \leq \Flat(R_{i(x)} \trace B(x,\rho) - S_{x,\rho}) + \sum_{\underset{j \neq i(x)}{j=1}}^{N} \Mass(R_{j} \trace B(x,\rho)) \\
&\overset{\eqref{altro_caso}}{\leq} \Flat(R_{i(x)} \trace B(x,\rho) - \tilde{R}_{i} ) + \Flat(\tilde{R}_{i} - S_{x,\rho}) + 2 \e \Mass(R_{i(x)} \trace B(x,\rho)) \\
&\overset{\eqref{stima1}, \eqref{stima2}}{\leq} C \e \Mass(R \trace B(x,\rho)).  
\end{split}
\end{equation}
This proves \eqref{eqn:flat-blow-up}.
\end{proof}

A straightforward iteration argument yields the following corollary.

\begin{corollary} \label{lemma_qo}
Let $R = \llbracket E, \tau, \theta \rrbracket$ be a rectifiable $m$-current in $\R^{n}$. Then, for $\Haus^{m}$-a.e. $x \in E$ 
\begin{equation}\label{limite}
\lim_{r \to 0} \frac{\Flat(R \trace B(x,r) - S_{x,\rho})}{\Mass(R \trace B(x,r))} = 0.
\end{equation}
\end{corollary}
\begin{proof}
For every $i\in \N$ define the set $E_i$ to be the set $E'$ given by Lemma \ref{lemma_tecnico} applied to $R$ with $\e=2^{-i-1}$, and let $F_i\subset E_i$ be the set of Lebesgue points of ${\bf 1}_{E_i}$ (inside $E_i$) with respect to $\theta\Haus^m\trace E$.
By \cite[Corollary 2.23]{AFP}, the set $F_i$ equals the set $E_i$ up to a set of $\Haus^m$-measure $0$ and for every $x\in F_i$ and for $\rho$ sufficiently small (possibly depending on $x$) it holds
\begin{equation*}
\begin{split}
\Mass(R \trace B(x,\rho)- R \trace (E_i \cap B(x,\rho))) &= 
 \int_{(E\setminus E_i)\cap B(x,\rho)} \theta \, d\Haus^m  
 \\&\leq 2^{-i-1} \int_{E\cap B(x,\rho)} \theta\, d\Haus^m = 2^{-i-1} \Mass(R \trace B(x,\rho)).
\end{split}
\end{equation*}
Hence by Lemma \ref{lemma_tecnico} for every $x \in F_i$ there exists $r_i(x)>0$ such that for every $0<\rho< r_i(x)$
\begin{equation*}
\begin{split}
\Flat(R \trace B(x,\rho) - S_{x,\rho})&\leq \Mass(R \trace B(x,\rho) - R \trace (E_i \cap B(x,\rho)))+ \Flat(R \trace (E_i \cap B(x,\rho)) - S_{x,\rho})\\
&\leq 2^{-i} \Mass(R \trace B(x,\rho))
\end{split}
\end{equation*}
and 
$$\Mass(R \trace (E \setminus F_i)) \leq 2^{-i-1}.$$
Denoting $F:= \bigcup_{i\in \N}\bigcap_{j\geq i}F_j$, and noticing that $E \setminus F= E \cap F^c = E \cap \bigcap_{i\in \N}\bigcup_{j\geq i}F^c_j $ is contained in $\bigcup_{j\geq i}F^c_j$ for every $i\in \N$, we have
$$\Mass(R\trace(E\setminus F)) \leq \lim_{i \to \infty}\sum_{j=i}^\infty \Mass(R\trace(E\setminus F_{j}))\leq \lim_{i \to \infty}\sum_{j=i}^\infty \frac{1}{2^j}=0$$
and this implies that $\Haus^m(E\setminus F)=0$.
Since every $x \in F$ belongs definitively to every $F_j$ (namely, for every $x \in F$ there exists $i_0(x)\in \N$ such that $x\in F_i$ for every $i \geq i_0(x)$), we obtain \eqref{limite}.
\end{proof}

\begin{proof}[Proof of Proposition \ref{thm:poly_approx}] Let $R$ be represented by $R = \llbracket E, \tau, \theta \rrbracket$ with $\theta \in L^{1}(\Haus^{m} \trace E ; \left( 0, \infty \right))$.
We denote  
$$\mu:=\theta \Haus^m\trace E.$$
Moreover, if $\MM(R)<+\infty$, we define the positive finite measure
$$ \nu:=H(\theta) \Haus^m\trace E.$$
Fix $\e>0$.
We make the following

{\bf{Claim}:} There exists a finite family of mutually disjoint balls $\{B_i\}_{i=1}^N$ with $B_i:=B(x_i,r_i)$, such that the following properties are satisfied:
\begin{itemize}
\item[$(i)$] $$r_{i} \leq \e \qquad \forall \, i=1,\dots,N \qquad \mbox{and} \qquad \mu(\R^n \setminus (\cup_{i=1}^N  B_i))\leq \e;$$
\item[$(ii)$] if we denote $R_{i} := R \trace B_{i}$ and $S_{i} := S_{x_{i}, r_{i}}$, then $$ \Flat(R_{i} - S_{i}) \leq \e \mu(B_{i}); $$
\item[$(iii)$] $$|\mu(B_i)-\theta(x_i) \omega_m r_i^m|\leq \e \mu(B_i), \qquad \forall \, i=1,\dots, N;$$
\item[$(iv)$] if $\MM(R)<+\infty$, then
$$ H(\theta(x_i)) \omega_m r_i^m\leq (1+\e)\nu(B_i), \qquad \forall \, i=1,\dots, N.$$
\end{itemize} 

Let us for the moment assume the validity of the claim and see how to conclude the proof of the proposition. 

By point $(iii)$ in the claim we deduce 
\begin{equation}\label{stimamassa}
\Mass(S_{i}) \leq (1 + \e) \Mass(R_{i}).
\end{equation}
and by point  $(iv)$ we get
\begin{equation}\label{stimaH}
\MM(S_{i}) \leq (1 + \e) \MM(R_{i}).
\end{equation}
On the other hand, we can find a polyhedral chain $P_i \in \mathbf P_m(\R^n)$ (supported on $\pi_i \cap B_i$, $\pi_{i} := \pi_{x_{i}}$), such that
\begin{equation}\label{pezzoIIII}
\Flat(P_i-S_i)\leq \e \mu(B_i), \qquad \MM(P_i)\leq \MM(S_i) \qquad \mbox{and} \qquad \Mass(P_i)\leq \Mass(S_i).
\end{equation}
Indeed, it is enough to approximate the $m$-dimensional current $S_i$ with simplexes with constant multiplicity and supported in $B_{i} \cap \pi_{i}$. 

To conclude, we denote $P:=\sum_{i=1}^NP_i$ and we estimate
\begin{equation}\label{pezzoIV}
\begin{split}
\Flat(R-P)&\leq \sum_{i=1}^N\Flat(R_i-P_i)+\Mass(R\trace (\R^n \setminus (\cup_{i=1}^N  B_i))) \\
 &\overset{(i)}{\leq} \e +\sum_{i=1}^N\Flat(R_i-S_i)+\sum_{i=1}^N\Flat(S_i-P_i)\overset{(ii),\eqref{pezzoIIII}}{\leq} \e+2\sum_{i=1}^N \e \mu(B_i)\leq \e + 2\e \Mass(R).
\end{split}
\end{equation}

Moreover
\begin{equation}\label{pezzoV}
\MM(P)= \sum_{i=1}^N \MM(P_i) \overset{\eqref{pezzoIIII}}{\leq}  \sum_{i=1}^N \MM(S_i)\overset{\eqref{stimaH}}{\leq} (1+\e)\sum_{i=1}^N \MM(R_i)\leq (1+\e)\MM(R)
\end{equation}
and
\begin{equation}\label{pezzoVI}
\Mass(P)= \sum_{i=1}^N \Mass(P_i) \overset{\eqref{pezzoIIII}}{\leq}  \sum_{i=1}^N \Mass(S_i)\overset{\eqref{stimamassa}}{\leq}(1+\e)  \sum_{i=1}^N \Mass(R_i)\leq (1+\e)\Mass(R).
\end{equation}

{\bf{Proof of the Claim}:} Consider the set $F$ of points $x\in E$ such that the following properties hold:
\begin{enumerate}
\item $x$ satisfies $$ \lim_{r\to 0} \frac{\Flat(R \trace B(x,r) - S_{x,r})}{\Mass(R \trace B(x,r))} = 0; $$
\item denoting $\eta_{x,r}:\R^n \to \R^n$ the map $y \mapsto \frac{y-x}{r}$, we have the following convergences of measures for $r \to 0$:
\begin{equation}\label{weakuno}
\mu_{x,r}:= r^{-m} (\eta_{x,r})_\#(\mu\trace B(x,r)) \rightharpoonup \theta(x)\Haus^m\trace ((x+\mbox{span}(\tau(x)))\cap B(0,1)),
\end{equation}
and
\begin{equation}\label{weakdue}\nu_{x,r}:= r^{-m} (\eta_{x,r})_\#(\nu\trace B(x,r)) \rightharpoonup H(\theta(x))\Haus^m\trace ((x+\mbox{span}(\tau(x)))\cap B(0,1)).
\end{equation}
\end{enumerate}
We observe that properties (1) and (2) hold for $\mu$-a.e. point. Indeed the fact that (1) holds for $\mu$-a.e. $x$ follows from Corollary \ref{lemma_qo}, while the fact that (2) holds for $\mu$-a.e. $x$ is a consequence of \cite[Theorem 4.8]{DL}. Moreover, by \eqref{weakuno} and by \eqref{weakdue}, for every $x\in F$ there exists a radius $r(x)<\e$ such that 
$$
|\mu_{x,r}(B(0,1))-\theta(x) \omega_m|\leq \frac{\e}2\theta(x) \omega_m,  \qquad \mbox{for a.e. } r<r(x).
$$
This inequality implies that
\begin{equation}\label{intermedia}
|\mu(B(x,r))-\theta(x) \omega_m r^m|\leq \frac{\e}2 \theta(x) \omega_m r^m, \qquad \mbox{for a.e. } r<r(x),
\end{equation}
so that in particular
$$\theta(x)\left(1-\frac{\e}{2}\right) \omega_m r^m \leq \mu(B(x,r)), \qquad \mbox{for a.e. } r<r(x).
$$
Plugging the last inequality in the right-hand side of \eqref{intermedia}, we get
$$|\mu(B(x,r))-\theta(x) \omega_m r^m|\leq \frac{\e}{2-\e}\mu(B(x,r))\leq \e \mu(B(x,r)), \qquad \mbox{for a.e. } r<r(x).$$
which gives condition $(iii)$ of the Claim.

Analogously, we get that
$$|\nu(B(x,r))-H(\theta(x)) \omega_m r^m|\leq \e \nu(B(x,r)), \qquad \mbox{for a.e. } r<r(x).$$ 

The validity of the claim is then obtained via the Vitali-Besicovitch covering theorem (\cite[Theorem 2.19]{AFP}).

\end{proof}

\section{Proof of Proposition~\ref{thm:normal}}
We first observe that the condition \eqref{ratio} is necessary for the validity of \eqref{non-rect}. Indeed, consider a map $H$ as in Assumption \ref{hp:funzione_H} for which \eqref{ratio} does not hold. It means that there exists a constant $C>0$ and a sequence $\{\theta_i\}_{i\in \N}$ converging to $0$ such that $H(\theta_i)\leq C\theta_i$ for every $i \in \N$.
We consider now the sequence of polyhedral $m$-chains $\{P_i\}_{i\in \N}$ supported in the unit cube $[0,1]^{n}$ and  defined as 
$$P_i:=\sum_{j=1}^{N_i}\llbracket \pi_i^j\cap [0,1]^n,\tau,\theta_i\rrbracket,$$
where for $i$ fixed, $\pi_i^j$ are $m$-planes parallel to $\{x_{m+1}=\ldots=x_n=0\}$ whose last $(n-m)$ coordinates are ``uniformly distributed'' in $[0,1]^{n-m}$,  $\tau$ is a fixed orientation for all the $m$-planes $\pi_i^j$ not depending on $i$ or $j$ and $N_i:=\min\{N\in \N: N\theta_i \geq 1\}$.
Since $\theta_i\to 0$, then $N_i\to \infty$. For $i$ large enough, so that $\theta_i N_i\leq 2$, we can compute
$$\Phi_{H}(P_{i})=\sum_{j=1}^{N_i} \Phi_{H}(\llbracket \pi_i^j\cap [0,1]^n,\tau,\theta_i\rrbracket)=N_iH(\theta_i)\leq CN_i\theta_i \leq 2C.$$
Nevertheless, since $\theta_iN_i\to 1$, then the sequence $\{P_i\}_{i\in \N}$ converges in flat norm to the $m$-current $T$, acting on $m$-forms as
$$
\langle T, \omega \rangle = \int_{[0,1]^n} \langle \omega(x), \tau \rangle \, d\mathcal{L}^n(x),
$$
which belongs to $(\F_m(\R^n)\cap  \lbrace T \in \mathscr{E}_{m}(\R^{n}) \, \colon \, \Mass(T) < \infty \rbrace)\setminus \Rc_m(\R^n)$. Clearly, $F_{H}(T) \leq 2C$.\\

We show now that, if $H$ is also monotone non-decreasing on $\left[ 0, \infty \right)$, then condition \eqref{ratio} is also sufficient to the validity of \eqref{non-rect}. The proof is a consequence of the definition of $F_H$ in \eqref{eq:funzionale_F} and the following Lemma (see also \cite[Lemma 4.5]{CDRM}):
\begin{lemma}\label{lemma:limit-rect}
Assume $H$ is as in Assumption \ref{hp:funzione_H}, is monotone non-decreasing on $\left[ 0, \infty \right)$, and satisfies \eqref{ratio}. Let $\{R_j\}_{j \in \mathbb{N}}\subset \mathbf{R}_m(\R^n)$ and let us assume that
$$
\sup_{j\in \N}\Mass_H( R_j)\leq C<+\infty.
$$
If $\lim_{j\to \infty} \Flat
(R_j - T ) = 0$ for some $T \in \mathbf{F}_m(\R^n)$ with finite mass, then $T$ is in fact rectifiable. 
\end{lemma}
\begin{proof}
{\it Step 1}. We prove the lemma for $m=0$, recalling  that a $0$-dimensional rectifiable current $R=\llbracket E,\tau,\theta \rrbracket$, with $\tau(x)=\pm 1$, is an atomic signed measure (i.e. a measure supported on a countable set). 

We observe that \eqref{ratio} implies that there exists $\delta_0>0$ such that $H(\theta)>0$ for every $\theta \in (0,\delta_0)$.
We define the monotone non-decreasing function $f:[0,\delta_0)\to [0,+\infty)$ given by
$$
f( \theta):= 
\begin{cases}
\sup_{t \in (0,\theta]} \displaystyle \frac{t}{H(t)} & \mbox{ if } 0 < \theta < \delta_{0},\\
0 & \mbox{ if } \theta=0.
\end{cases}
$$
By assumption \eqref{ratio}, $f$ is continuous in $0$ and $H(\theta)f(\theta) \geq \theta$. Fix any $\delta \in (0,\delta_0)$. For any $j\in \N$
\begin{equation*}
\begin{split}
\Mass (R_j \trace &\{ x: \theta_j(x)<\delta \} ) 
= \int_{E_j \cap \{\theta_j<\delta \} } \theta_j(x) \, d \Haus^m(x)
\leq  \int_{E_j \cap \{\theta_j<\delta \} } f(\theta_j(x))H(\theta_j(x)) \, d \Haus^m(x)\\&
\leq f(\delta)\int_{E_j \cap \{\theta_j<\delta \} } H(\theta_j(x)) \, d \Haus^m(x) \leq f(\delta) {\Mass_H(R_j)} \leq C f(\delta).
\end{split}
\end{equation*}
Therefore, up to subsequences the sequence $\{R_j \trace\{ x: \theta_j(x)<\delta \}\}_{j\in\N}$ converges to a signed measure $R_2$ of mass less than or equal to $C f(\delta)$. On the other hand, using the upper bound on $\MM(R_{j})$ and the monotonicity of $H$, we deduce that the measures $R_j \trace \{ x: \theta_j(x)\geq\delta \} $ are supported on a uniformly (with respect to $j$) bounded number of points, and converge to a discrete measure $R_1$. Hence, for any $\e>0$, the limit $T$ can be written as the sum of a discrete measure $R_{1}$ and of an error $R_{2}$ with mass less than or equal to $\e$. Since $\e$ is arbitrary, the statement follows.

{\it Step 2}. We prove the claim for $m>0$. 

We apply \cite[4.3.1]{FedererBOOK} to the sequence $\{R_j\}_{j\in\N}$ to deduce that for any $I \in I(n,m)$
$$
\lim_{j\to \infty} \int_{\R^m} \Flat(\langle R_j-T, p_I, y\rangle) \, dy \leq  \lim_{j\to \infty} \Flat (R_j-T) = 0.
$$
Since the sequence of non-negative functions $\{ \Flat(\langle R_j-T, p_I, \cdot \rangle)\}_{j \in \N}$ converges in $L^1(\R^m)$ to $0$, up to a (not relabelled) subsequence, we get the pointwise convergence
$$\lim_{j \to \infty} \Flat(\langle R_j-T, p_I, y\rangle) = 0 \qquad \mbox{for }\Haus^{m}\mbox{-a.e. } y\in \R^m, \mbox{ for every } I \in I(n,m).$$
We apply the Fatou lemma and \cite[Corollary 3.2.5(5)]{depauwhardt} to the sequence $\{R_j\}_{j\in\N}$ to deduce 
\begin{equation}\label{eqn:mass-alpha-slices}
\int_{\R^m} \liminf_{j\to \infty} \Mass_H(\langle R_j, p_I, y\rangle) \, dy 
\leq \liminf_{j\to \infty}
\int_{\R^m} \Mass_H(\langle R_j, p_I, y\rangle) \, dy \leq \liminf_{j\to \infty} \Mass_H(R_j) \leq C. 
\end{equation}
Hence the integrand in the left-hand side is finite a.e., namely 
$\liminf_{j\to \infty} \Mass_H(\langle R_j, p_I, y\rangle)  < \infty $ for $\Haus^{m}\mbox{-a.e. } y\in \R^m$, for every $I \in I(n,m)$.
Hence we are can apply Step 1 to a.e.\ slice $\langle R_j, p_I, y\rangle$ to a $y$-dependent subsequence and deduce that
\begin{equation}\label{white}
\langle T, p_I, y\rangle \mbox{ is $0$-rectifiable for }\Haus^{m}\mbox{-a.e. } y\in \R^m, \mbox{ for every } I \in I(n,m).
\end{equation}
To conclude the proof we employ Theorem~\cite[Rectifiable slices theorem, pp. 166-167]{white_rect}, which ensures that a finite mass flat chain $T$ is rectifiable if and only if property \eqref{white} holds.
\end{proof}

\begin{samepage}

%
%
%

\Addresses

\end{samepage}

\end{document}